\documentclass[reqno,colorlinks=true,citecolor=magenta,linkcolor=blue]{amsart}

\usepackage{hyperref,times} 
\usepackage{mathrsfs}
\usepackage[numbers,compress]{natbib}
\usepackage{amssymb}

\newcommand{\be}{\begin{equation}\begin{aligned}}
\newcommand{\ee}{\end{aligned}\end{equation}}
\newcommand{\ben}{\begin{equation}\nonumber\begin{aligned}}

 \newcommand{\A}{\mathscr{A}}

 \newcommand{\R}{\mathbb{R}}

\newcommand{\N}{\mathbb{N}}

\renewcommand{\d}{{\rm d}}

\newcommand{\dist}{{\rm dist}}

\renewcommand{\leq}{\leqslant}
\renewcommand{\geq}{\geqslant}
\newcommand{\xto}{\xrightarrow}

\begin{document}

 \title[$(L^2,L^\gamma\cap H_0^1)$-continuity of reaction-diffusion equations]{Strong $(L^2,L^\gamma\cap H_0^1)$-continuity in initial data of nonlinear reaction-diffusion equation in any space dimension 
 }

\author[H. Cui, P.E. Kloeden $\&$ W. Zhao]{}

 \keywords{Smoothing property; global attractor; regularity; fractal dimension. \newline
 \emph{\text{\ \ \quad E-mails. }} h.cui@outlook.com;
 kloeden@na-uni.tuebingen.de; gshzhao@sina.com}

\maketitle

\centerline{\scshape Hongyong Cui}
\smallskip
{\footnotesize 
 \centerline{School of Mathematics and Statistics, 
 Huazhong University of Science and Technology}
 \centerline{Wuhan 430074, China}}
 \medskip

\centerline{\scshape Peter E. Kloeden}
\smallskip
{\footnotesize
 \centerline{Mathematisches Institut, Universit\"at T\"ubingen}
 \centerline{D-72076 T\"ubingen, Germany}}
 
 \medskip
\centerline{\scshape Wenqiang Zhao}
\smallskip
{\footnotesize
 \centerline{School of Mathematics and Statistics,
 Chongqing Technology and Business University}
 \centerline{Chongqing 400067, China}}

\begin{abstract} 
 In this paper, we study the continuity in initial data of a classical reaction-diffusion equation with arbitrary $p>2$ order nonlinearity and in any space dimension $N\geq 1$. It is proved that the weak solutions can be $(L^2, L^\gamma\cap H_0^1)$-continuous in initial data for any $\gamma\geq 2$ (independent of the physical parameters of the system), i.e., can converge in the norm of any $L^\gamma\cap H_0^1$ as the corresponding initial values converge in $L^2$. Applying this to the global attractor we find that, with external forcing only in $ L^2$, the attractor $\A$ attracts bounded subsets of $L^2$ in the norm of any $L^\gamma\cap H_0^1$, and that every translation set $\A-z_0$ of $\A$ for any $z_0\in \A$ is a finite dimensional compact subset of $L^\gamma\cap H_0^1$.  The main technique we employ is a combination of the mathematical induction and a decomposition of the nonlinearity by which the continuity result is strengthened to $(L^2, L^\gamma \cap H_0^1)$-continuity   and, since interpolation inequalities are avoided, the restriction   on space dimension is removed.

\end{abstract}

\numberwithin{equation}{section}
\newtheorem{theorem}{Theorem}[section]
\newtheorem{lemma}[theorem]{Lemma}
\newtheorem{proposition}[theorem]{Proposition}
\newtheorem{corollary}[theorem]{Corollary}

\theoremstyle{definition}
\newtheorem{definition}[theorem]{Definition}

\theoremstyle{remark}
\newtheorem{remark}[theorem]{Remark}

 \tableofcontents

\section{Introduction}

 The continuity problem of solutions is definitely of significance for the study of evolution equations. As described by Evans \cite[p7]{Evans10}, ``(the continuity property) is particularly important for problems arising from physical applications: we would prefer that our (unique) solution changes only a little when the conditions specifying the problem change a little.'' In addition, the continuity property is often important for further studies of a dynamical system, e.g., for studying the regularity of global attractors \cite{cui15na,zhong06jde,cui18jdde}, constructing an exponential attractor and estimating its fractal dimensions \cite{miranville08chapter,shirikyan13spdeac}, and studying the stability of the attractor under perturbations, etc. Hence, in case the continuity result of a system is not satisfactory people have to find alternative conditions to carry out further studies. For example, the norm-to-weak continuity condition \cite{zhong06jde}, quasi strong-to-weak continuity condition \cite{cui18jdde,li08jde} and closed-graph condition \cite{pata07cpaa,zelati15jma} were introduced in various studies. Nevertheless,
 even for these cases where continuity condition can be less crucial, better continuity condition will facilitate the analysis. Hence, it is always worth a deeper study even if some continuity results have already been known under certain conditions. 
 
 In this paper we consider the following classical reaction-diffusion equation on bounded smooth domain $D\subset \R^N$ with $N\in \N$:
\be
 \frac{\d u}{\d t} -\triangle u +f(u) =g(x) ,& \label{eq0} \\
 u(0)=u_0,\quad u|_{\partial D} =0,\ &
\ee 
where $g\in L^2(D)$ and the nonlinearity $f:\R\to \R$ is a $C^1$-function satisfying some dissipative conditions, say the odd degree polynomial 
\[
 f(s)=\sum_{j=1}^{p-1} a_js^j,
\]
where $p>2$ is even and $a_{p-1}>0$. The well-known PDE result says for every initial value $u_0\in L^2(D)$ the initial-boundary problem \eqref{eq0} has a unique weak solution $u$ satisfying 
\[
 u\in C([0,\infty); L^2(D)),\quad u\in L^p_{loc}(0,\infty;L^p(D))\cap L_{loc}^2(0,\infty;H_0^1(D)) ,
\]
and $u$ continuously depends on $u_0$ in $L^2(D)$. Moreover, Robinson \cite[p227]{robinson01} argued that, ``without further restrictions on $p$ we cannot prove, for general $N$, that the map $u_0\mapsto u(t)$ is continuous (in $H_0^1(D)$), although we can prove this for $N\leq 3$.'' In other words, the strong continuity in $H_0^1 $ remained unknown for general $p>2$ and $N\geq 1$. 

 In 2008, by making using of interpolation inequalities Trujillo $\&$ Wang \cite{trujillo08na} gave a solution for the problem of Robinson. More precisely, by estimating the uniform boundedness of $tu(t)$ in $L^\infty (0,T;H^2)$ and by the continuous embedding $\|w\|_{H^1} \! \leq c\|w\|_{H^2}^{1/2} \|w\|^{1/2}$ Trujillo $\&$ Wang \cite{trujillo08na} obtained the {$(H_0^1\cap L^p, H^1_0)$-continuity} of strong solutions for all $p>2$ and $N\geq1$, where by $(X,Y)$-continuity we mean that the solutions converge in the topology of $Y$ as the initial data converge in the topology of $X$. 
 Note that a drawback of the techniques employed in \cite{trujillo08na} is the dependence on derivatives w.r.t. $t$ in both sides of \eqref{eq0}, so they do not apply to stochastic evolution equations since general stochastic processes are not differentiable.

 Then in 2015, a mathematical induction method was proposed by Cao, Sun $\&$ Yang \cite{cao15jde} where the time-derivatives were avoided and the $(H_0^1\cap L^p, H^1_0)$-continuity result was proved for the stochastic system with additive Brownian noise. This method was then further improved by Zhu $\&$ Zhou \cite{zhu16cma} in a deterministic and unbounded domain case by which the continuity result of the reaction-diffusion equation was improved to a much stronger $(L^2, H_0^1)$-continuity. However, since the analysis of \cite{cao15jde,zhu16cma} relies so heavily on interpolation inequalities, the analysis there is only for dimension $N\geq 3 $ and does not apply directly to all $N\geq 1$, especially in unbounded domains. 
 
 The restriction on space dimension is a natural cost of interpolation inequalities, so it is meaningful if there is a way to bypass them. Most recently, Zhao \cite{zhao17jmaa}, in a study of a stochastic $p$-Laplacian equation on $\R^N$, dramatically modified the induction method of Cao et al. \cite{cao15jde} by appending the original equation with a second nonlinear term $\tilde f$ which was assumed with certain satisfactory conditions. With the auxiliary term $\tilde f$ the interpolation inequalities were avoided so the result can hold for all $N\geq 1$, but this method itself greatly changes the structure of the nonlinearity of the equation. 
 
 In this paper, we present a decomposition method of the nonlinearity $f$ to establish a stronger $(L^2,L^\gamma\cap H_0^1)$-continuity of \eqref{eq0} for any $p>2$ and $N\geq 1$, where $\gamma\geq 2$ is arbitrary and independent of the physical parameters of the system. The key idea is that, instead of introducing an auxiliary nonlinear term as in \cite{zhao17jmaa}, we prove that the original nonlinearity $f$ can be decomposed into two: one provides good properties leading to the desired continuity results, and the other remains controllable. This technique avoids both time-derivatives and interpolation inequalities, so can apply to stochastic equations (which will be shown in our future work) and has no restrictions on space dimension {$N\geq 1$}. It is proved that the weak solutions of the reaction-diffusion equation can be {$(L^2, L^\gamma\cap H_0^1)$-continuous and even $(L^2, L^\gamma\cap H_0^1)$-smoothing} for all $\gamma\geq 2$ (independent of all the physical parameters of the system), and the solutions $u(t)$ are shown to be bounded in $ L^\infty(\varepsilon,\infty;L^{p })$ for any $\varepsilon>0$ rather than only bounded in $L^{p}_{loc}(0,\infty;L^p)$ as usually understood \cite{robinson01,zhao12na}.

 Then we apply the main techniques as well as the new continuity result to attractor theory. It is shown that, with the external forcing only in $L^2$, the global attractor of \eqref{eq0} in $L^2$ is a compact set in $L^{p}\cap H_0^1$, and pullback attracts bounded sets in $L^2$ under the topology of $L^\gamma\cap H_0^1$ for any $\gamma\geq 2$, i.e., under a topology much more regular than the attractor itself can be. Moreover, the translation $\A-z_0:=\{a-z_0:a\in \A\}$ of the global attractor about any point $z_0\in \A$, e.g., $z_0$ a stationary solution, is shown to be a compact subset of any $L^\gamma\cap H_0^1$, $\gamma\geq 2$. In addition, making use of the new $(L^2, L^\gamma\cap H_0^1)$-smoothing property the upper bounds of the fractal dimensions of $\A$ and $\A-z_0$ in $L^p\cap H_0^1$ and in $L^\gamma\cap H_0^1$, respectively, are easily obtained.
 
 Note that though in the present paper we work in a deterministic, autonomous and bounded domain framework in order to keep the main idea clear, the method applies to non-autonomous, unbounded domain and even stochastic PDEs, which will be illustrated in our future work.

\section{The reaction-diffusion equation}
 \subsection{Settings}

 In this paper, we consider the following classical reaction-diffusion equation on some bounded smooth domain $D\subset \R^N$ with $N\in \N$:
\be
\frac{\d u}{\d t} +\lambda u-\triangle u +f(u) =g(x) , \\
 u(0)=u_0,\quad u|_{\partial D} =0,\ \label{eq} 
\ee 
where $\lambda>0$ is a fixed number, $g\in L^2(D)$ and the nonlinearity $f:\R\to \R$ is a $C^1$-function satisfying the following standard conditions
\begin{gather}
 f'(s) \geq \kappa |s|^{p-2}-l, \label{f1} \\
 f(s)s \geq \alpha |s|^p - \beta, \label{f2} \\
 |f(s)| \leq \sigma |s|^{p-1} + \sigma, \label{f3}
\end{gather}
where $p> 2$, and $l, \kappa, \alpha, \beta, \sigma$ are all positive constants. 

Notice that, if condition \eqref{f2} is satisfied for some $\alpha>0$, then it holds
 automatically for all numbers that smaller than $\alpha$. Hence, it is not restrictive at all to let 
\be \label{kappa}
 \alpha \leq \frac{
\kappa }{p-1} .
\ee

An example of such a nonlinearity $f$ is an odd degree polynomial
\[
 f(s)=\sum_{j=1}^{2k-1} b_j s^j,
\]
where $k >1$ and $b_{2k-1} >0$. 
 In this example, $p=2k$ is even, and, generally, $\kappa$ and $\alpha$ are in the form $\kappa= (2k-1)b_{2k-1}-(2k-1)\varepsilon_1$, $\alpha=b_{2k-1}-\varepsilon_2$, where $\varepsilon_1$ and $\varepsilon_2$ are flexible coefficients from Young's inequality so that can be chosen as $\varepsilon_1\geq \varepsilon_2$ to make \eqref{kappa} satisfied. Note that in the case of $ b_j\equiv 0$ for $j=2,\cdots, 2k-2$ and $b_1<0$, which gives $\varepsilon_1=0$ and $\varepsilon_2\neq 0$, we have \eqref{kappa} with strict $<$.
 
 In the sequel, we often omit the domain $D$ and write, e.g., $L^\gamma(D)$ as $L^\gamma$ for any $\gamma \geq 2$. The norms $\|\cdot\|_{L^\gamma}$ are written as $\|\cdot\|_\gamma$ and $\|\cdot\|:=\|\cdot\|_2$.

 \subsection{$L^\infty(\varepsilon,\infty;L^\gamma)$-estimates of solutions}
 
 Generally, the regularity of solutions depends heavily on that of the external forcing $g$. The following lemma indicates a clear relationship between the integrability of solutions and that of $g$.

 \begin{lemma} \label{lem8.1}
 Under condition \eqref{f2}, for any $\varepsilon>0$ there exists a family of positive constants $\big\{C^{(k)}_\varepsilon\big\}_{k\in \N}$ such that the solution $u$ of \eqref{eq} satisfies 
 \begin{equation*}
 \begin{aligned}
 &
 \| u(t) \|^{p a_k}_{p a_k} 
 \leq C^{(k)}_\varepsilon \left(e^{-\lambda t} \|u_0\|^2 
 + \|g \|^{\frac{pa_{k+1} }{p-1}}_{\frac{pa_{k+1} }{p-1}} +1\right),\quad t>\varepsilon,
 \end{aligned}
 \end{equation*} 
 where 
 \[
 a_1 =1,\quad a_{k +1}=a_{k}+\frac{p-2}p , \quad k\in \N.
 \]

 \end{lemma}
 \begin{remark} \rm
 Lemma \ref{lem8.1}
 indicates that with $g\in L^2$ the solutions belong to $ L^\infty(\varepsilon,\infty; L^p)$ for any $\varepsilon>0$ rather than only to $L^p_{loc}(0,\infty;L^p)$ as usually understood, see, e.g., \cite{robinson01,zhao12na}. 
 \end{remark}
 \begin{proof} [ Proof of Lemma \ref{lem8.1}] Without loss of generality, let $\varepsilon\in (0,1)$. For $t> \varepsilon $ we prove by mathematical induction a stronger result, that there exists a family of positive constants $\big\{C^{(k)}_\varepsilon\big\}_{k\in \N}$ such that the solution $u(t)$ satisfies 
 \begin{equation*}
 \begin{aligned}
 &
 \frac \varepsilon2 \| u(t) \|^{p a_k}_{p a_k} + \alpha \int_{\sum_{j=1}^k\! \frac \varepsilon {3^{j}}}^{\varepsilon}
 \int^t_r e^{\lambda (s-t)}\|u(s)\|^{pa_{k+1}}_{pa_{k+1}}\ \d s\d r\\
 &\quad 
 \leq C^{(k)}_\varepsilon \left(e^{-\lambda t} \|u_0\|^2 
 + \|g \|^{\frac{pa_{k+1} }{p-1}}_{\frac{pa_{k+1} }{p-1}} +1\right) ,\quad k\in\N .
 \end{aligned}
 \tag{$G_k$}
 \end{equation*}

 Multiply \eqref{eq} by $u$ and integrate over $D$, by \eqref{f2}, to obtain
 \ben
 \frac 12 \frac{\d }{\d t} \|u\|^2 +\lambda \|u\|^2 +\|\nabla u\|^2 +\alpha \|u\|_p^p-\beta |D| \leq c\|g \|^2 + \frac\lambda 2\|u\|^2 ,
 \ee
so
 \be \label{16.3}
 \frac{\d }{\d t} \|u\|^2 +\lambda \|u\|^2 + \|u\|_p^p\leq c\|g \|^2 + c ,
 \ee
 where and throughout the paper $c$ is a generic constant that may change its value from line to line. 
 Multiplying \eqref{16.3} by $e^{\lambda t}$ and integrating over $(0,t)$ we have
 \be
 \|u(t )\|^2 +\int^t_0 e^{\lambda (s-t)}\|u(s)\|^p_p\ \d s 
 \leq e^{-\lambda t}\|u_0\|^2 +c\|g \|^2 
 +c .
 \ee
 This implies that 
 \be \label{15.2}
 \int^{1}_0 e^{ -\lambda }\|u(s)\|^p_p\ \d s 
 & \leq \int^{1}_0 e^{\lambda(s-1 ) }\|u(s)\|^p_p\ \d s \\
 &
 \leq e^{-\lambda }\|u_0\|^2 +c\|g \|^2 +c .
 \ee

 Multiply \eqref{eq} by $|u|^{p-2}u$ and integrate over $D$ to obtain, by \eqref{f2},
 \ben
 \frac 1p \frac{\d }{\d t} \|u\|^p_p +\lambda \|u\|^p_p +\alpha \|u\|_{2p-2}^{2p-2} -\beta \|u\|_{p-2}^{p-2} \leq c\|g \|^2 + \frac\alpha p\|u\|^{2p-2}_{2p-2} ,
 \ee 
 so, since $\|u\|_{p-2}^{p-2} \leq \eta \|u\|_{2p-2}^{2p-2} +c$ for any $\eta>0$,
 \be \label{15.1}
 \frac{\d }{\d t} \|u\|^p_p + \lambda \|u\|^p_p +\alpha \|u\|_{2p-2}^{2p-2} 
 \leq c\|g \|^2 + c .
 \ee 
 Multiply \eqref{15.1} by $e^{\lambda t}$ and then integrate over $(r,t)$ for $r\in(0, \varepsilon)$ to obtain 
 \be \label{15.11}
 \|u(t)\|^p_p +\alpha \int^t_r e^{\lambda (s-t)}\|u(s)\|^{2p-2}_{2p-2}\ \d s 
 & \leq e^{-\lambda (t-r)}\|u(r)\|^{p}_{p}+c\|g \|^2 
 +c \\
 & 
 \leq e^{-\lambda (t-1)}\|u(r)\|^{p}_{p} +c\|g \|^2 +c .
 \ee 
 Integrating the above inequality w.r.t. $r$ over $(0, \varepsilon)$, by \eqref{15.2} we have 
 \be
 &
 \varepsilon \|u(t)\|^p_p +\alpha \int_0^{\varepsilon} \int^t_r e^{\lambda (s-t)}\|u(s)\|^{2p-2}_{2p-2}\ \d s \d r \\
 &\quad 
 \leq e^{ \lambda -\lambda t } \int_0^{1} 
 \|u(r)\|^{p}_{p}\ \d r + c \|g \|^2 +c \\
 &\quad 
 \leq e^{\lambda -\lambda t} \|u_0\|^2 +c\|g \|^2 +c ,
 \ee
 which concludes $(G_k)$ for $k=1$.

 Next, assuming $(G_k)$ holds we prove $(G_{k+1})$.
 
 Multiplying \eqref{eq} by $ u | u|^{pa_{k+1}-2 } $ and then integrating over $D$, we have 
 \ben
 \frac{1}{pa_{k+1} }\frac{\d}{\d t} \| u\|_{p a_{k+1}}^{p a_{k+1}} 
 + \lambda \| u\|_{pa_{k+1}}^{pa_{k+1}} +\int f(u) u | u|^{pa_{k+1}-2 } \ \d x \leq \int g u | u|^{pa_{k+1}-2 } \ \d x ,
 \ee
 and then, by \eqref{f2} and Young's inequality,
 \ben
 & \frac{1}{pa_{k+1} }\frac{\d}{\d t} \| u\|_{p a_{k+1}}^{p a_{k+1}} 
 + \lambda \| u\|_{pa_{k+1}}^{pa_{k+1}} +\int ( \alpha|u|^p-\beta) | u|^{pa_{k+1}-2 } \ \d x\\
 &\quad \leq \int g u | u|^{pa_{k+1}-2 } \ \d x \leq c\|g \|_{\frac{pa_{k+1}+p-2}{p-1}}^{\frac{pa_{k+1}+p-2}{p-1}} +\frac \alpha{pa_{k+1}}\|u\|^{pa_{k+1}+p-2}_{pa_{k+1}+p-2} .
 \ee
 Hence, since $\|u\|_{pa_{k+1}-2}^{pa_{k+1}-2} \leq \eta \|u\|_{pa_{k+1}+p-2}^{pa_{k+1}+p-2} +c$ for any $\eta>0$,
 \ben
 \frac{\d}{\d t} \| u\|_{p a_{k+1}}^{p a_{k+1}} 
 + \lambda \| u\|_{pa_{k+1}}^{pa_{k+1}} +\alpha \|u\|^{pa_{k+1}+p-2}_{pa_{k+1}+p-2} 
 \leq c\|g \|_{\frac{pa_{k+1}+p-2}{p-1}}^{\frac{pa_{k+1}+p-2}{p-1}} +c ,
 \ee
 i.e., with $pa_{k+2}=pa_{k+1}+p-2 $, 
 \be \label{16.1}
 \frac{\d}{\d t} \| u\|_{p a_{k+1}}^{p a_{k+1}} 
 + \lambda \| u\|_{pa_{k+1}}^{pa_{k+1}} + \alpha \|u\|^{pa_{k+2} }_{pa_{k+2} } 
 \leq c\|g \|_{\frac{pa_{k+2} }{p-1}}^{\frac{pa_{k+2} }{p-1}} +c .
 \ee 
 Multiply \eqref{16.1} by $e^{\lambda t}$ and then integrate over $(r,t)$ for $r\in(0,\varepsilon)$ to obtain 
 \be \label{16.6}
 &
 \|u(t)\|^{p a_{k+1}}_{p a_{k+1}} +\alpha \int^t_r e^{\lambda (s-t)}\|u(s)\|^{p a_{k+2}}_{p a_{k+2}}\ \d s \\
 &\quad 
 \leq e^{-\lambda (t-r)}\|u(r)\|^{{p a_{k+1}}}_{{p a_{k+1}}}
 +c\int^t_r e^{\lambda (s-t)}\|g\|^{\frac{pa_{k+2} }{p-1}}_{\frac{pa_{k+2} }{p-1}} \ \d s +c
 \\ &\quad \leq e^{ \lambda (r-t)}\|u(r)\|^{p a_{k+1}}_{p a_{k+1}} +c \|g \|^{\frac{pa_{k+2} }{p-1}}_{\frac{pa_{k+2} }{p-1}} +c 
 .
 \ee
 Integrating \eqref{16.6} with respect to $r$ over $(\rho, \varepsilon )$ for $\rho\in\big (\sum_{j=1}^{k}\frac \varepsilon{3^j},\sum_{j=1}^{k+1}\frac \varepsilon{3^j} \big)$, since $\varepsilon-\rho \geq \frac{\varepsilon}2 $, we have 
 \ben
 & \frac \varepsilon 2 \|u(t)\|^{p a_{k+1}}_{p a_{k+1}} +\alpha \int_{\sum_{j=1}^{k+1}\frac \varepsilon{3^j}}^{\varepsilon} \int^t_r e^{\lambda (s-t)}\|u(s)\|^{p a_{k+2}}_{p a_{k+2}}\ \d s \d r\\
 &\quad \leq \left(\varepsilon-\rho \right)
 \|u(t)\|^{p a_{k+1}}_{p a_{k+1}} +\alpha \int_{\rho}^{ \varepsilon} \int^t_r e^{\lambda (s-t)}\|u(s)\|^{p a_{k+2}}_{p a_{k+2}}\ \d s \d r\\ 
 & \quad
 \leq \int_{\rho}^ \varepsilon e^{\lambda (r-t)} \|u(r )\|^{p a_{k+1}}_{p a_{k+1}}\ \d r
 +c \|g \|^{\frac{pa_{k+2} }{p-1}}_{\frac{pa_{k+2} }{p-1}} +c .
 \ee
 Integrating w.r.t. $\rho$ over $ \big(\sum_{j=1}^{k }\frac \varepsilon{3^j},\sum_{j=1}^{k+1}\frac \varepsilon{3^j} \big)$, by $(G_k)$ we have 
 \ben
 & \frac \varepsilon{3^{k+1}} \left( \frac \varepsilon2 \|u(t)\|^{p a_{k+1}}_{p a_{k+1}} +\alpha \int_{\sum_{j=1}^{k+1}\frac \varepsilon{3^j}}^{\varepsilon} \int^t_r e^{\lambda (s-t)}\|u(s)\|^{p a_{k+2}}_{p a_{k+2}}\ \d s \d r \right) \\
 & \quad 
 \leq \int_{\sum_{j=1}^{k}\frac \varepsilon{3^j}}^{\sum_{j=1}^{k+1}\frac \varepsilon{3^j} } \! \int_{\rho}^\varepsilon
 e^{\lambda (r-t)} \|u(r)\|^{p a_{k+1}}_{p a_{k+1}}\ \d r \d \rho 
 +c \|g \|^{\frac{pa_{k+2} }{p-1}}_{\frac{pa_{k+2} }{p-1}} +c\\
 &\quad \leq C^{(k)}_\varepsilon\left(e^{-\lambda t} \|u_0\|^2 
 + \|g \|^{\frac{pa_{k+1} }{p-1}}_{\frac{pa_{k+1} }{p-1}} +
 1\right) +c \|g \|^{\frac{pa_{k+2} }{p-1}}_{\frac{pa_{k+2} }{p-1}} +c,
\ee
by which $(G_{k+1})$ is concluded.
 \end{proof}

 \section{$(L^2, L^\gamma\cap H_0^1)$-continuity of solutions in initial data}

 In this section, we prove the $(L^2, L^\gamma\cap H_0^1)$-continuity of weak solutions of \eqref{eq} for any $\gamma\geq 2$.
 \begin{definition}[$(X,Y)$-continuity] \label{def_bi}\rm
 Suppose that $X$ and $Y$ are two Banach spaces. A mapping $\mathcal M:X\to X $ is said to be $(X,Y)$-continuous if $\mathcal M(x_1)-\mathcal M(x_2)\in Y$ for any $x_1, x_2\in X$ and $\|\mathcal M(x_n)-\mathcal M(x)\|_Y\to 0$ for any convergent sequence $x_n\to x$ in $X$.
\end{definition}
 Note that by Definition \ref{def_bi} an $(X,Y)$-continuous mapping need not take values in $Y$, but the difference of any two values belongs to $Y$.
 
 Let $u_j$, $j=1,2$, be the unique weak solutions of \eqref{eq} corresponding to initial data $u_{0,j}$ from $L^2$, respectively. Then the difference $\bar u:=u_1-u_2$ satisfies 
 \be 
 \frac{\d \bar u}{\d t} +\lambda \bar u-\triangle \bar u +f( u_1) -f(u_2)=0,
 \label{eq_diff} 
 \\
 \bar u(0)=\bar u_0= u_{0,1}-u_{0,2}. \qquad\
 \ee

 \subsection {$(L^2, L^\gamma)$-continuity}

 With $g$ only in $ L^2$, by Lemma \ref{lem8.1} the solutions of \eqref{eq} are expected at most in $L^p$. However, we will see that the difference of any two solutions will belong to any high order $L^\gamma$, $\gamma \geq 2$, and, moreover, the system is $(L^2, L^\gamma)$-continuous in a H$\ddot{\rm o}$lder way.

 We begin with a decomposition of the nonlinear term $f$, from which we obtain some new but crucial conditions, without requiring any additional assumptions.
 
 \begin{lemma}[Decomposition of the nonlinear term] \label{c2}
 Any $C^1$-function $f$ with conditions \eqref{f1}-\eqref{kappa} can be decomposed as $f=f_1+f_2$, where $f_1$ and $f_2$ are both $C^1$-functions for which there exist positive coefficients 
 $\{\alpha_1, \sigma_1, \kappa_2, l_2, \alpha_2, \beta_2, \sigma_2\}$ such that 
 $f_1$ satisfies
 \begin{gather}
 \big(f_1(s_1)-f_1(s_2)\big)(s_1-s_2) \geq \alpha_1 |s_1-s_2|^p , \label{f11} \\
 | f_1(s_1)-f_1(s_2) |\leq \sigma_1 |s_1-s_2| \big (1+|s_1|^{p-2}+|s_2|^{p-2}\big) , \label{f12}
 \end{gather}
 and $f_2$ inherits all the properties \eqref{f1}-\eqref{kappa} from $f$, satisfying
 \begin{gather}
 f_2'(s) \geq \kappa_2|s|^{p-2} -l_2 , \label{f21} \\
 f_2(s)s \geq \alpha_2 |s|^p - \beta_2, \label{f22} \\
 |f_2(s)| \leq \sigma_2 |s|^{p-1} +\sigma_2 , \label{f23} \\
 \alpha_2\leq \frac{\kappa_2}{p-1} . \label{f24}
 \end{gather} 
 \end{lemma}
 
 \begin{proof} 
 We prove the proposition by constructing a proper $f_1$. 
 Let
 \be
 f_1(s):= \frac{\alpha}2 |s|^{p-2}s - \sigma,\quad s\in\R,
 \ee 
 where $\alpha, \sigma>0$ are constants given
 in \eqref{f2} and \eqref{f3}. 
 Then such defined $f_1$ satisfies \eqref{f11} and \eqref{f12}. To see this, let us first recall from \cite{Bartsch04jde} that
there exist positive constants $c_1,\cdots , c_4$ such that for all $\xi,\eta\in \R^N$
\begin{gather}
\left| |\xi|^{p-2}\xi- |\eta|^{p-2}\eta \right| \leq c_1(|\xi|+|\eta|)^{p-2} |\xi-\eta|, \label{B1}\\ 
\left(|\xi|^{p-2}\xi -|\eta|^{p-2}\eta\right) \cdot (\xi-\eta)\geq c_4 |\xi-\eta|^p, \quad \text{for } p>2. \label{B4}
\end{gather} 
 Therefore, by \eqref{B4},
 \ben
 \big(f_1(s_1)-f_1(s_2)\big)(s_1-s_2) & = \frac{\alpha}2 \big( |s_1|^{p-2}s_1 -|s_2|^{p-2}s_2 \big)(s_1-s_2) \\
 &\geq \frac{\alpha}2 c_4 |s_1-s_2|^p;
 \ee
 and, by \eqref{B1},
 \ben
 \big | f_1(s_1)-f_1(s_2) \big| & = \frac{\alpha}2 \big| |s_1|^{p-2}s_1 -|s_2|^{p-2}s_2 \big| \\
 &\leq \frac{\alpha}2 c_1 (|s_1|+|s_2|)^{p-2}|s_1-s_2| \\
 &\leq c |s_1-s_2| \big (1+|s_1|^{p-2}+|s_2|^{p-2}\big) .
 \ee
 
 Next we show that 
 $ f-f_1=:f_2$ satisfies \eqref{f21}-\eqref{f24}. 
 
 Since $f_1'(s)= \frac{\alpha}2 (p-1)|s|^{p-2} $, by \eqref{f1}, $$f'(s)-f_1'(s)\geq \Big(\kappa- \frac{\alpha}2 (p-1)\Big)|s|^{p-2}-l , $$
 where $\kappa- \frac{\alpha}2 (p-1) =:\kappa_2$ is positive because of \eqref{kappa}, and thereby \eqref{f21} follows.

 Since $f_1(s)s= \frac{\alpha}2 |s|^{p}- \sigma s \leq \frac {3 \alpha}4|s|^p +c $ and $f$ satisfies \eqref{f2},
 \ben
 \big (f(s)-f_1(s)\big)s \geq \alpha |s|^p-\beta -\Big(\frac{3 \alpha} 4 |s|^p +c\Big) =\frac\alpha 4 |s|^p-\beta-c ,
 \ee
 and, as $f$ satisfies \eqref{f3} and $|f_1(s)| \leq \frac{\alpha}2 |s|^{p-1} + \sigma$,
 \ben
 |f(s)-f_1(s)| & \leq \sigma |s|^{p-1}+ \sigma + \frac{\alpha}2 |s|^{p-1} + \sigma =\Big ( \sigma+ \frac{\alpha}2 \Big)|s|^{p-1} 
 +2\sigma .
 \ee
 Therefore, \eqref{f22} and \eqref{f23} hold for $f-f_1$. 

By construction $\kappa_2=\kappa- \frac{\alpha}2(p-1)$ and $\alpha_2=\alpha/4$, \eqref{f24} follows
from \eqref{kappa}.
 \end{proof}
 
 For later convenience we conclude the following corollary from Lemma \ref{c2}.
 \begin{corollary} \label{c3}
 Any $C^1$-function $f$ with conditions \eqref{f1}-\eqref{kappa} has the property 
 \[
 \big( f(s_1)-f(s_2)\big)(s_1-s_2) |s_1-s_2|^{r} \geq \alpha_1 |s_1-s_2|^{p+r} -l_2 |s_1-s_2|^{r+2} 
 \]
 for any $r\geq 0$ and $ s_1, s_2\in \R$, where $\alpha_1$ and $l_2$ are positive constants in Lemma \ref{c2}.
 \end{corollary}
 \begin{proof} Making use of the decomposition $f=f_1+f_2$, by \eqref{f11} and \eqref{f21} we have 
 \ben
 \big( f(s_1)-f(s_2)\big)(s_1-s_2) |s_1-s_2|^{r} &= \big( f_1(s_1)-f_1(s_2)\big)(s_1-s_2) |s_1-s_2|^{r} \\
 &\quad +
 \big( f_2(s_1)-f_2(s_2)\big)(s_1-s_2) |s_1-s_2|^{r} \\
 & \geq \alpha_1 |s_1-s_2|^{p+r} +f'_2(\xi) |s_1-s_2|^{r+2}\\
 & \geq \alpha_1 |s_1-s_2|^{p+r} +
 (\kappa_2 |\xi|^{p-2} -l_2) |s_1-s_2|^{r+2}\\
 & \geq \alpha_1 |s_1-s_2|^{p+r} -l_2 |s_1-s_2|^{r+2} . 
 \ee
 \end{proof}

 \begin{theorem}[$(L^2,L^\gamma)$-continuity] \label{lem}
 Let conditions \eqref{f1}-\eqref{kappa} hold and $T> 0$. Then there exists a family of positive constants $\big\{ C_{T}^{(k)}\big\}_{k\in \N}$, where each $ C_{T}^{(k)}$ depends exclusively on $T$ and parameters $\{\mu, \alpha_1,p\}$, such that the difference $\bar u=u_1-u_2$ of solutions corresponding to any initial data in $ L^2$ satisfies 
 \begin{equation*}
 t\big\|t^{b_k}\bar u(t)\big\|^{p a_k}_{p a_k} \leq C_{T}^{(k)} \|\bar u_0\|^2,\quad t\in(0,T], \tag{$A_k$}
 \end{equation*}
 and 
 \begin{equation*}
 \int_0^T\big\| s^{b_{k+1}} \bar u(s)\big\|_{p a_{k+1}}^{ p a_{k+1}}\ \d s \leq C_{T}^{(k)}\|\bar u_0\|^2, \tag{$B_k$}
 \end{equation*} 
 where 
 \[
 a_1=b_1=1,\quad a_{k +1}=a_{k}+\frac{p-2}p,\quad b_{k+1}=\frac{a_k b_k }{ a_{k+1}} +\frac{ 2}{p a_{k+1}} ,\quad k\in \N.
 \]
 \end{theorem}
 \begin{remark}\rm \label{rem}
 {The above lemma implies an arbitrary $(L^2,L^\gamma)$-smoothing property of the system: for any $\gamma \geq 2$ there exists a constant $c_\gamma$ such that }
 \[
 \|\bar u(1)\|_{\gamma}^{\gamma} \leq c_\gamma \|\bar u_0\|^2 .
 \]
 Indeed, for any $ \gamma \geq 2$ there exists a $k\in \N$ such that $\gamma\in [2,pa_k)$, so $\|\bar u(1)\|_\gamma^\gamma \leq \max\{\| \bar u(1)\|^2, \|\bar u(1)\|_{pa_k}^{pa_k}\}$ and the remark follows from Theorem \ref{lem}.
 \end{remark}
 \begin{proof}[Proof of Theorem \ref{lem}] The proof is done by induction. 
 We begin with $(A_1)$ and $(B_1)$.
 Multiplying \eqref{eq_diff} by $\bar u$ and then integrating over $D$ we have 
 \ben
 \frac 12 \frac{\d }{\d t} \|\bar u\|^2 +\lambda \|\bar u\|^2 +\|\nabla \bar u\|^2 +\int \bar u\big(f(u_1)-f(u_2)\big) \ \d x =0.
 \ee
 By Corollary \ref{c3} we have 
 \ben
 \int \bar u\big(f(u_1)-f(u_2)\big) \d x & \geq \alpha_1 \|\bar u\|_p^p -l_2 \|\bar u\|^2 .
 \ee
 Hence, with $\mu:= \max\{2(l_2-\lambda),1\} \geq 1$,
 \be \label{20.1}
 \frac{\d }{\d t} \|\bar u\|^2 + 2\alpha_1 \|\bar u\|^p_p +2\|\nabla \bar u\|^2 
 \leq \mu \|\bar u\|^2 .
 \ee 
 By Gronwall's lemma it follows
 \be \label{12.6} 
 \|\bar u(t) \|^2 \leq e^{\mu t}\| \bar u_0\|^2,\quad \forall t\in (0,T],
 \ee
 so, integrating \eqref{20.1} over $(0,T)$ gives
 \be\label{12.7}
 2\alpha_1 \int^T_0 \|\bar u(s)\|_p^p \ \d s +2 \int^T_0 \|\nabla \bar u(s)\|^2 \ \d s 
 &\leq \int^T_0 \mu e^{\mu s}\|\bar u_0\|^2\ \d s+\|\bar u_0\|^2 \\
 &= \big(e^{\mu T} +1\big) \|\bar u_0\|^2 .
 \ee
 This implies that there exists a positive constant $ C_{T}$ exclusively depending on $T $ and parameters $\{\mu, \alpha_1,p\}$ such that 
 \begin{align} 
 \int_0^T \Big( \| s \bar u(s)\|_p^p + \| s \nabla \bar u(s)\|^2\Big) \ \d s \nonumber 
 & 
 \leq \left( T^{ p}+ T^{2 }\right) \int_0^T\! \Big( \|\bar u(s)\|_p^p + \|\nabla \bar u(s)\|^2\Big) \ \d s \\
 & 
 \leq C_{T} \|\bar u_0\|^2 .\label{21.3}
 \end{align}
 
 Multiplying \eqref{eq_diff} by $|\bar u|^{p-2}\bar u$ and integrating over $D$ we have 
 \ben
 \frac 1p \frac{\d }{\d t} \|\bar u\|_p^p+\lambda \|\bar u\|_p^p -\int \triangle \bar u \big( |\bar u|^{p-2}\bar u\big) \d x +\int\big (f(u_1)-f(u_2) \big)|\bar u|^{p-2} \bar u \ \d x=0.
 \ee
 Since 
 \ben
 -\int \triangle \bar u \big( |\bar u|^{p-2}\bar u\big)\ \d x =\int \nabla \bar u \cdot \nabla \big( |\bar u|^{p-2}\bar u\big) \ \d x \geq 0 
 \ee
 and, by Corollary \ref{c3},
 \ben
 \int\big (f(u_1)-f(u_2) \big)|\bar u|^{p-2} \bar u \ \d x 
 \geq \alpha_1 \|\bar u\|_{2p-2}^{2p-2}-l_2 \|\bar u\|_p^p,
 \ee
 it follows that 
 \be \label{21.1}
 \frac 1p \frac{\d }{\d t} \|\bar u\|_p^p + \alpha_1 \|\bar u\|_{2p-2}^{2p-2} \leq (l_2-\lambda) \|\bar u\|_p^p .
 \ee
 Note that, for all $r>0$, 
 \be \label{ab}
 \frac{\d}{\d t}\big \| t^r \bar u\big\|_p^p &= \frac{\d}{\d t}\Big( t^{rp} \| \bar u\|_p^p \Big) \\
 &= rp t^{rp-1} \| \bar u\|_p^p + t^{rp} \frac{\d}{\d t} \| \bar u\|_p^p .
 \ee
 Hence, multiplying \eqref{21.1} by $ t^{ p}$ we obtain 
 \ben &
 \frac1p \frac{\d}{\d t} \| t \bar u\|_p^p - t^{ p-1} \| \bar u\|_p^p 
 + \alpha_1 t^{ p} \|\bar u \|_{2p-2}^{2p-2} \leq (l_2-\lambda) t^{ p} \| \bar u\|_p^p,
 \ee
 that is, 
 \be \label{21.4}
 \frac{\d}{\d t} \big\| t \bar u\big\|_p^p + p\alpha_1 \left\| t^{\frac{ p}{2p-2}} \bar u \right\|_{2p-2}^{2p-2} 
 & \leq p \left((l_2-\lambda) +\frac{ 1}{t} \right)\big \| t \bar u\big\|_p^p\\
 & \leq c \left(1 +\frac{ 1}{t} \right)\big \| t \bar u\big\|_p^p,
 \ee
 and then
 \ben
 t \frac{\d}{\d t} \big\| t \bar u\big\|_p^p 
 \leq c \left( {t} +1\right)\big \| t \bar u\big\|_p^p ,
 \ee 
 where $c=c(p,l_2,\lambda)>0$ is a constant. Integrating the above inequality over $(0,t)$ for $t\in (0,T]$, we have 
 \be \label{28.7}
 c ( {T} +1 )\int_0^t \big \| s \bar u(s)\big\|_p^p\ \d s &\geq \int_0^t s \frac{\d}{\d s} \big\| s \bar u(s) \big\|_p^p \ \d s \\
 & = t \big\| t \bar u(t)\big\|_p^p -\int_0^t \big\| s \bar u(s) \big\|_p^p \ \d s,
 \ee
 where the identity is by integration by parts. Then, \eqref{28.7} and \eqref{21.3} give 
 \be \label{21.5}
 t \big\| t \bar u(t)\big\|_p^p & \leq c ( {T} +1 )\int_0^t \big \| s \bar u(s)\big\|_p^p\ \d s \\
 & \leq C_{T} \|\bar u_0\|^2,\quad t\in (0,T].
 \ee

 Multiplying \eqref{21.4} by $ t^2$, by \eqref{21.5} we have 
 \be \label{21.6}
 t^2 \frac{\d}{\d t} \big\| t \bar u\big\|_p^p + \Big\| t^{\frac{ p+2}{2p-2}} \bar u \Big\|_{2p-2}^{2p-2} 
 & \leq c t^2 \left(1 +\frac{ 1}{t} \right)\big \| t \bar u\big\|_p^p\\
 & \leq C_{T} \|\bar u_0\|^2.
 \ee
 Then integrating \eqref{21.6} over $(0,T)$ and by integration by parts we obtain
 \ben
 & T^2 \big\| T \bar u(T)\big\|_p^p -\int_0^T 2 s \big\| s \bar u(s)\big\|_p^p \ \d s +\int_0^T \Big\| s^{\frac{ p+2}{2p-2}}\bar u (s) \Big\|_{2p-2}^{2p-2} \ \d s\leq C_{T} \|\bar u_0\|^2,
 \ee
 which gives, with $b_2:={\frac{ p+2}{2p-2}}$, $a_2:=\frac{2p-2}p$ and by \eqref{21.5}, 
 \be 
 \int_0^T \big\| s^{b_2} \bar u (s) \big\|_{pa_2}^{pa_2} \ \d s&\leq \int_0^T 2 s \big\| s \bar u(s)\big\|_p^p \ \d s + C_{T} \|\bar u_0\|^2 \\
 &\leq 2 T C_{T} \|\bar u_0\|^2 + C_{T} \|\bar u_0\|^2.\label{21.7}
 \ee
 By \eqref{21.5} and \eqref{21.7} we have proved $(A_k) $ and $(B_k) $ for $k=1$.

Next, for $k\geq 1$, assuming $(A_k)$ and $(B_k)$ we prove $(A_{k+1})$ and $(B_{k+1})$.
 Multiplying \eqref{eq_diff} by $ \bar u |\bar u|^{pa_{k+1}-2 } $ and then integrating over $D$, we have 
 \ben
 & \frac{1}{pa_{k+1} }\frac{\d}{\d t} \|\bar u\|_{p a_{k+1}}^{p a_{k+1}} 
 + \lambda \|\bar u\|_{pa_{k+1}}^{pa_{k+1}} +\int \big(f(u_1)-f(u_2) \big) \bar u |\bar u|^{pa_{k+1}-2 } \ \d x \leq 0. 
 \ee
 Since, by Corollary \ref{c3} again, 
 \[
 \int \big(f(u_1)-f(u_2) \big) \bar u |\bar u|^{pa_{k+1}-2 } \ \d x \geq 
 \alpha_1 \|\bar u\|_{pa_{k+1}+p-2}^{pa_{k+1}+p-2}
 -l_2 \|\bar u\|_{pa_{k+1}}^{pa_{k+1}},
 \] 
 we have 
 \be \label{28.1}
 \frac{\d}{\d t} \|\bar u\|_{p a_{k+1}}^{p a_{k+1}} 
 + \|\bar u\|_{pa_{k+1}+p-2}^{pa_{k+1}+p-2}
 \leq c\|\bar u\|_{pa_{k+1}}^{pa_{k+1}} ,
\ee
where $c=c(pa_{k+1},\alpha_1,l_2)>0$. Similarly to \eqref{ab} we have 
 \ben
 \frac{\d}{\d t}\big \| t^{b_{k+1}} \bar u\big\|_{pa_{k+1}}^{p a_{k+1}}
 = b_{k+1} pa_{k+1} t^{ b_{k+1} pa_{k+1}-1} \| \bar u\|_{pa_{k+1}}^{pa_{k+1}} 
 + t^{ b_{k+1} pa_{k+1}} \frac{\d}{\d t} \| \bar u\|_{pa_{k+1}}^{pa_{k+1}},
 \ee
 so multiplying \eqref{28.1} by $ t^{ b_{k+1} pa_{k+1}} $ we obtain 
 \be \label{28.3}
 & \frac{\d}{\d t}\big \| t^{b_{k+1}} \bar u\big\|_{pa_{k+1}}^{p a_{k+1}} + t^{ b_{k+1} pa_{k+1}} \|\bar u\|_{pa_{k+1}+p-2}^{pa_{k+1}+p-2}\\
 &\quad \leq c \big (1+ t^{-1}\big ) \big \| t^{ b_{k+1}}\bar u\big \|_{pa_{k+1}}^{pa_{k+1}} ,
 \ee
 and then for all $t\in (0,T)$ 
 \be \label{28.2} 
 & t \frac{\d}{\d t}\big \| t^{b_{k+1}} \bar u\big\|_{pa_{k+1}}^{p a_{k+1}} \leq c ( T-1 ) \big \| t^{ b_{k+1}}\bar u\big \|_{pa_{k+1}}^{pa_{k+1}} ,
\ee
 where $c=c(pa_{k+1},b_{k+1}, \alpha_1,l_2)>0$. Integrating \eqref{28.2} over $ (0,t)$ we have
 \ben
 &
 t \big \| t^{b_{k+1}} \bar u(t)\big\|_{pa_{k+1}}^{p a_{k+1}} 
 -\int_0^t \big \| s^{b_{k+1}} \bar u(s)\big\|_{pa_{k+1}}^{p a_{k+1}} \ \d s \\
 &\quad \leq c ( T-1 ) \int_0^t \big \| s^{ b_{k+1}}\bar u(s)\big \|_{pa_{k+1}}^{pa_{k+1}} \ \d s,\quad t\in(0,T],
 \ee
 by which $(A_{k+1})$ is concluded since we have assumed $(B_k)$.
 
 Then we prove $(B_{k+1})$. With $a_{k+2}:= a_{k+1}+\frac{p-2}p$, \eqref{28.3} is reformulated as 
 \ben 
 \frac{\d}{\d t}\big \| t^{b_{k+1}} \bar u\big\|_{pa_{k+1}}^{p a_{k+1}} +\Big \| t^{\frac{ b_{k+1} a_{k+1}}{a_{k+2}}} \bar u \Big \|_{pa_{k+2}}^{pa_{k+2}} 
 \leq c \big (1+ t^{-1}\big ) \big \| t^{ b_{k+1}}\bar u\big \|_{pa_{k+1}}^{pa_{k+1}} ,
 \ee
 which multiplied by $ t^2$ gives
 \be \label{28.6}
 & t^2 \frac{\d}{\d t}\big \| t^{b_{k+1}} \bar u\big\|_{pa_{k+1}}^{p a_{k+1}} + t^2\Big \| t^{\frac{ b_{k+1} a_{k+1}}{a_{k+2}}} \bar u \Big \|_{pa_{k+2}}^{pa_{k+2}}\\
 &\quad \leq c t ( T +1 ) \big \| t^{ b_{k+1}}\bar u\big \|_{pa_{k+1}}^{pa_{k+1}}, \quad \forall t\in(0,T).
 \ee 
 With $b_{k+2}:=\frac{b_{k+1}pa_{k+1}+2}{pa_{k+2}}$, the second term of \eqref{28.6} is rewritten as 
 \ben
 t^2\Big \| t^{\frac{ b_{k+1} a_{k+1}}{a_{k+2}}} \bar u \Big \|_{pa_{k+2}}^{pa_{k+2}} = \big \| t^{b_{k+2}} \bar u \big \|_{pa_{k+2}}^{pa_{k+2}}.
 \ee
 Hence, from \eqref{28.6} and $ (A_{k+1})$ it follows 
 \be \label{28.5}
 t^2 \frac{\d}{\d t}\big \| t^{b_{k+1}} \bar u\big\|_{pa_{k+1}}^{p a_{k+1}} 
 + \big \| t^{b_{k+2}} \bar u \big \|_{pa_{k+2}}^{pa_{k+2}} \leq c t ( T+1 ) \big \| t^{ b_{k+1}}\bar u\big \|_{pa_{k+1}}^{pa_{k+1}} 
 \leq C_{T} \|\bar u_0\|^2.
 \ee 
 Integrating \eqref{28.5} over $(0,T)$ we have 
 \ben
 & T^2 \big \| T^{b_{k+1}} \bar u(T)\big\|_{pa_{k+1}}^{p a_{k+1}} -\int_0^T 2 s\big \| s^{b_{k+1}} \bar u(s)\big\|_{pa_{k+1}}^{p a_{k+1}} \ \d s +\int_0^T
 \big \| s^{b_{k+2}} \bar u(s) \big \|_{pa_{k+2}}^{pa_{k+2}}\ \d s \\
 &\quad \leq C_{T} \|\bar u_0\|^2,
 \ee 
 so, by $(B_k)$,
 \ben
 \int_0^T
 \big \| s^{b_{k+2}} \bar u(s) \big \|_{pa_{k+2}}^{pa_{k+2}}\ \d s & \leq 2 T \int_0^T \big \| s^{b_{k+1}} \bar u(s) \big\|_{pa_{k+1}}^{p a_{k+1}} \ \d s + C_{T} \|\bar u_0\|^2\\
 & \leq \Big( 2 T C_{T}^{(k)} + C_{T} \Big) \|\bar u_0\|^2 ,
 \ee 
 from which $(B_{k+1})$ follows.
 \end{proof}
 
 \subsection{\texorpdfstring{$(L^2,H^1_0)$}--continuity}
 
 Now, 
 we study the $(L^2,H_0^1)$-continuity of system \eqref{eq}. As has been noted in introduction, though the continuity in $H_0^1$ was also studied in \cite{zhu16cma,cao15jde} in a framework of non-autonomous and random dynamical systems, see also \cite{zhong06jde,sun10jde}, the analysis here is quite different. Thanks to our $(L^2,L^\gamma)$-continuity established previously we do not rely heavily on interpolation 
 inequalities and the continuity in $H_0^1$ is obtained directly for all space dimension $N\geq 1$. 
 
 As in \cite{zhu16cma,cao15jde}, we assume that for some positive constant $c$ 
 \be \label{f_add}
 |f(s_1)-f(s_2)| \leq c |s_1-s_2|(1+|s_1|^{p-2}+|s_2|^{p-2}) .
 \ee
 Since $f\in C^1$, it is equivalent to require positive constants $\kappa_0$ and $l_0$ such that 
 \ben
 | f'(s) | \leq \kappa_0 |s|^{p-2} + l_0 .
 \ee

 \begin{theorem}[$(L^2,H_0^1)$-continuity] \label{lem2}
 Let conditions \eqref{f1}-\eqref{kappa} and \eqref{f_add} hold. Then for any $t>0$ and initial data $u_{0,j} $ with $\|u_{0,j}\|\leq R$ ($j=1,2$) there exist positive constants $C_{R,t}$ and $C_t$ such that the difference $\bar u $ of the corresponding solutions of \eqref{eq} satisfies
 \ben
 \| \nabla \bar u(t) \|^2 \leq 
 C_{R,t} \|\bar u_0\| ^{\frac2{p-1}} 
 + C_{t} \| \bar u_0\|^2 ,
 \ee
 where $C_{R,t}$ and $C_t$ can be explicitly computed independently of space dimension $N\geq 1$.
 \end{theorem}
 \begin{remark}\rm \label{rem2}
 Since $p>2$, Theorem \ref{lem2} gives the $(L^2,H_0^1)$-smoothing property 
 \[
 \|\nabla \bar u(1)\| \leq c \|\bar u_0\|^{\frac 1 {p-1}} ,\quad \forall \|\bar u_0\|\leq 1.
 \]
 \end{remark}
 
 \begin{proof}[Proof of Theorem \ref{lem2}]
 Multiplying \eqref{eq_diff} by $-\triangle \bar u$ and integrating over $D$ we have 
 \ben
 \frac12 \frac{\d}{\d t}\|\nabla \bar u\|^2+\lambda \|\nabla \bar u\|^2 +\|\triangle \bar u\|^2 =\int \triangle\bar u\big(f(u_1)-f(u_2)\big) \ \d x .
 \ee
 Since by \eqref{f_add} and Young's inequality we have
 \ben
 \int \triangle\bar u\big(f(u_1)-f(u_2)\big) \d x 
 & 
 \leq c \int |\Delta u| |\bar u| \big(1+|u_1|^{p-2} +|u_2|^{p-2}\big) \d x \\ & 
 \leq \|\triangle \bar u\|^2 +c \int \big(|u_1|^{2p-4} +|u_2|^{2p-4}\big) |\bar u|^2 \ \d x + c\|\bar u \|^2\\ 
 & \leq \|\triangle \bar u\|^2 +c \big(\|u_1\|^{2p-4}_{2p-2} +\|u_2\|^{2p-4}_{2p-2} \big) \|\bar u\|_{2p-2}^2 + c\|\bar u \|^2 ,
 \ee
 it follows 
 \be \label{12.1}
 \frac{\d}{\d t}\|\nabla \bar u\|^2 \leq 
 c \big(\|u_1\|^{2p-4}_{2p-2} +\|u_2\|^{2p-4}_{2p-2} \big) \|\bar u\|_{2p-2}^2 + c\|\bar u \|^2 .
 \ee
 Take 
 \[
 r :=\frac{p+3}{2p-2} .
 \]
 Then multiplying \eqref{12.1} by $ t^{2r}$, by formula \eqref{ab} we have 
 \ben
 \frac{\d}{\d t} \| t^r \nabla \bar u\|^2 -2r t^{2r-1} \|\nabla \bar u\|^2 \leq 
 c t^{2r} \big(\|u_1\|^{2p-4}_{2p-2} +\|u_2\|^{2p-4}_{2p-2} \big) \|\bar u\|_{2p-2}^2 + c t^{2r}\|\bar u \|^2.
 \ee 
 For $s\in (\frac{t}2,t)$, integrating the above inequality over $(s,t)$ we obtain 
 \ben
 & \| t^r \nabla \bar u(t) \|^2 - \| s^r\nabla \bar u(s)\|^2 - \int_0^t 2r s^{2r-1} \|\nabla \bar u(s) \|^2 \ \d s \\
 &\quad \leq
 c \int^t_{\frac{t}2} s^{2r} \big(\|u_1(s) \|^{2p-4}_{2p-2} +\|u_2(s)\|^{2p-4}_{2p-2} \big) \|\bar u (s) \|_{2p-2}^2 \ \d s 
 + c\int_0^t s^{2r} \|\bar u(s) \|^2\ \d s,
 \ee
 and then integrating with respect to $s$ over $(\frac{t}2,t)$ yields 
 \be \label{12.3}
 & t\| t^{r} \nabla \bar u(t) \|^2 - \int_0^t\| s^r\nabla \bar u(s)\|^2\ \d s - t \int_0^t 2r s^{2r-1} \|\nabla \bar u(s) \|^2 \ \d s \\
 &\quad \leq
 c t \int^t_{\frac{t}2} s^{2r} \big(\|u_1(s) \|^{2p-4}_{2p-2} +\|u_2(s)\|^{2p-4}_{2p-2} \big) \|\bar u (s) \|_{2p-2}^2 \, \d s \\
 &\qquad\quad \
 + c t^{2r+1} \int_0^t \|\bar u(s) \|^2\ \d s .
 \ee
 
 Note that, by $(A_2)$ in Theorem \ref{lem} with $pa_2=2p-2$ and $pa_2b_2= p+2$, 
 \be \label{12.4}
 \sup_{s\in(0,t]}\Big( s^{2r } \|\bar u(s)\|^2_{2p-2}\Big) & = \sup_{s\in(0,t]} \Big( s\| s^{b_2}\bar u(s)\|^{pa_2}_{pa_2} \Big)^{\frac 2{pa_2}} \\
 &\leq \Big( C_{t}^{(2)} \|\bar u_0\|^2 \Big)^{\frac 1{p-1}} .
 \ee
 Hence, from \eqref{12.3} and \eqref{12.4} it follows
 \ben
 t ^{2r+1}\| \nabla \bar u(t) \|^2& \leq 
 c t \int^t_{\frac{t}2} \big(\|u_1(s) \|^{2p-4}_{2p-2} +\|u_2(s)\|^{2p-4}_{2p-2} \big) \d s\Big( C^{(2)}_{t} \|\bar u_0\|^2\Big)^{\frac 1{p-1}} \\
 &\ \quad 
 + c t^{2r+1} \! \int_0^t \|\bar u(s) \|^2\ \d s+ (2r+1) t^{2r} \! \int_0^t \|\nabla \bar u(s) \|^2 \ \d s ,
 \ee
 which along with \eqref{12.6} and \eqref{12.7} gives 
 \be \label{14.2}
 \| \nabla \bar u(t) \|^2 \leq 
 c_{t} \int^t_{\frac{t}2} \left (\|u_1(s) \|^{2p-4}_{2p-2} +\|u_2(s)\|^{2p-4}_{2p-2} \right) \d s \|\bar u_0\| ^{\frac2{p-1}} 
 + c_{t} \| \bar u_0\|^2,
 \ee
 where $c_{t}>0 $ is a constant depending on $C_{t}^{(2)}$ and $t$. 
 
 Recall that any solution $u_j$ ($j=1,2$) satisfies \eqref{15.1}, from which and analogously to \eqref{15.11} we have
 \ben 
 \int^t_{\frac t2} e^{\lambda (s-t)} \|u_j(s) \|^{2p-2}_{2p-2}\ \d s 
 & \leq e^{-\frac{\lambda t} 2} \Big \|u_j \Big(\frac t2 \Big) \Big\|^{p}_{p}+c\|g \|^2 
 +c \\
 &
 \leq c \|u_{0,j}\|^2+c\|g \|^2 +c =:C_R,
 \ee 
 where the second inequality is due to the uniform boundedness of $\|u_j(t)\|^p_p$ given in Lemma \ref{lem8.1} (taking $k=1$).
 Hence, 
 \ben
 \int^t_{\frac{t}2} \|u_j(s) \|^{2p-4}_{2p-2} \ \d s & \leq 
 \int^t_{\frac t2} \|u_j(s) \|^{2p-2}_{2p-2}\ \d s +ct 
 \leq C_Re^{\frac {\lambda t}2} 
 +c t 
 \ee 
 which along with \eqref{14.2} completes the proof. 
 \end{proof}

 \section{Applications to the global attractor}

 Recall that a global attractor $\A$ for a semigroup $S$ in a Banach space $X$ is a compact set in $X $ which is invariant under $S$, namely, $S(t,\A )=\A$ for all $t\geq 0$, and attracts all bounded subsets $B$ of $X$, namely, $\lim_{t\to \infty} \dist_X(S (t,B), \A)=0$ where $\dist_X$ denotes the Hausdorff semi-metric, see, e.g., \cite{babin92}.
 Under a standard argument as in \cite{robinson01,zelik00} it is well-known that the reaction-diffusion system \eqref{eq} with conditions \eqref{f1}-\eqref{f3} and $g\in L^2$ has a finite fractal dimensional global attractor in $L^2$. More precisely, we have
 \begin{lemma}\cite{robinson01,zelik00} \label{lem5}
 Let conditions \eqref{f1}-\eqref{f3} hold and $g\in L^2(D)$. Then the semigroup $S$ generated by the reaction-diffusion system \eqref{eq} has an absorbing set bounded in $H_0^1(D)$ and a global attractor $\A$ in $L^2(D)$ which has a finite fractal dimension $dim_F(\A;L^2(D)) <\infty$.
 \end{lemma}

 \subsection{Topological properties}
 In fact, according to the bi-spatial attractor theory one can show that the global attractor $\A$ is in fact compact in $L^p$ and in $ H_0^1$, and is attracting in the corresponding topology, see, e.g., \cite{zhao12na,cui18jdde,li15jde}. The key point is to prove the system to be asymptotically compact w.r.t$.$ the topology of $L^p$ and $H_0^1$, respectively, see, e.g., \cite[Theorem 3.9]{cui15na}. Since bi-spatial theory generally requires an absorbing ball that belongs to $L^p$ and $H_0^1$, at the light of Lemma \ref{lem8.1} one would not expect the attractor to be $(L^2, L^\gamma)$ for $\gamma>p$ and $g\in L^2$. Nevertheless, Sun \cite{sun10jde}, and then latter \cite{cao15jde,zhu16cma} in random and non-autonomous cases, showed that the attraction of the attractor can happen in $L^\gamma$ for any $\gamma\geq 2$, but only for $N\geq 3$ due to the restrictions of interpolation inequalities involved.

 In the following, making use of our $(L^2,L^\gamma)$-continuity we study the topological properties of the global attractor for all $N\geq 1$ in a different way from the bi-spatial attractor theory. We begin with some abstract analysis.
 
 Let $X,Y$ be two Banach spaces, and $S$ a semigroup on $X$ which need not take values in $Y$. The following result indicates that the $(X,Y)$-continuity ensures automatically more regular topological properties of an attractor, i.e., the attracting property and the compactness property. 
 
 \begin{proposition} \label{th-bi}
 Suppose that $S$ is a semigroup with global attractor $\A$ in $X$. If $S$ is moreover $(X,Y)$-continuous, that is, for any $t>0$ the mapping $S(t,\cdot)$ is $(X,Y)$-continuous satisfying Definition \ref{def_bi}, then 
 \begin{itemize}
 \item[(i)] the attractor $\A$ attracts bounded subsets of $X$ in the topology of $Y$;
 \item[(ii)] $\A$ is quasi compact in the topology of $Y$ in the sense that for any sequence $\{x_n\}_{n\in\N}\subset \A$, there exists a $b\in \A$ such that, up to a subsequence, 
 \[
 \|x_n-b\|_Y\to 0;
 \]
 if, moreover, $\A\subset Y$, then $\A$ is a compact subset of $Y$;
 \item[(iii)] if $ \A\cap Y$ is dense in $\A$, i.e., $\A = \overline{\A\cap Y}^X$, then $\A\subset Y$, and so
 $\A$ is a compact subset of $Y$; 
 \item[(iv)] for any $z_0\in \A$, the translation set $ \A-z_0=\{x-z_0:x\in \A\}$ of the attractor is a compact subset of $Y$. Consequently, if $0\in \A$ then $\A$ is a compact subset of $Y$.
 \end{itemize}
 \end{proposition}
 \begin{proof}
 (i) Given a bounded set $B\subset X$, we prove by contradiction that 
 \ben
 \dist_Y(S(t,B), \A)\to 0,\quad t\to \infty.
 \ee
 If it were not the case, then there exist a $\delta>0$ and a sequences $x_n\in B$ and $t_n\to \infty$ such that 
 \be \label{8.2}
 \dist_Y(S(t_n ,x_n), \A) \geq \delta, \quad n\in \N. 
 \ee
 
 Since $\A$ attracts $B$ in the topology of $X$ and is compact in $X$, there exists an $a\in \A$ such that, up to a subsequence,
 \ben
 \|S(t_n-1,x_n) -a\|_X\to 0.
 \ee
 Since $S$ is $(X,Y)$-continuous, this makes 
 \be \label{8.1}
 \|S(t_n,x_n)- S(1,a)\|_Y =\|S(1,S(t_n-1,x_n))- S(1,a)\|_Y \to 0.
 \ee
 Since $S (1,a)\in \A$ by the invariance of $\A$, \eqref{8.1} contradicts \eqref{8.2}.

 (ii) We prove that for any sequence $\{x_n\}_{n\in\N}\subset \A$, there exists a $b\in \A$ such that, up to a subsequence, 
 \[
 \|x_n-b\|_Y\to 0.
 \]
 By the invariance of $\A$ there exists a sequence $\{y_n\}_{n\in\N}\subset \A$ such that $x_n=S(1, y_n)$. 
 Since $\A$ is compact in $X$, there exists a $y\in \A$ such that, up to a subsequence, 
 \[
 \|y_n-y\|_X\to 0 ,
 \]
 which along with the $(X,Y)$-continuity of $S$ gives 
 \[
 \|x_n-S(1,y)\|_Y = \|S(1,y_n)-S(1,y)\|_Y\to 0.
 \]
 Noticing that $b:=S(1,y)\in \A$ by the invariance of $\A$, we have the result.
 
 (iii) To show that $\A$ is a compact subset of $Y$, by (ii) it suffices to prove that $\A\subset Y$. Let $\A|_Y:= \A\cap Y$. Then since $\A=\overline{ \A|_Y} ^X$, the proof will be concluded if we have $ \overline{ \A|_Y}^X=\overline{\A|_Y}^Y (\subset Y)$.
 Clearly, $ \overline{ \A|_Y}^X\supset \overline{ \A|_Y}^Y$. To prove $ \overline{ \A|_Y}^X\subset \overline{ \A|_Y}^Y$, take arbitrarily $a \in \overline{ \A|_Y}^X $. If $a\in \A|_Y$, then $a\in \overline{\A|_Y}^Y$ as desired. If $a\notin \A|_Y $, then there exists a sequence $a_n \in \A|_Y$ such that $a_n\xto Xa$. In addition, we have proved that, up to a subsequence, $ \|a_n-b\|_Y\to 0$ for some $b\in \A$, which means that $ a_n\xto Yb\in \overline{\A|_Y}^Y$. Therefore, by the uniqueness of a limit we have $a=b \in \overline{\A|_Y}^Y$.
 
 (iv) It is clear that $\A-z_0$ is quasi-compact in $Y$, so it suffices to prove $\A-z_0\subset Y$. Take arbitrarily a $y\in \A-z_0$, then we have $y=x-z_0$ for some $x\in \A$. By the invariance of $\A$, there exist $x_1, x_2\in \A$ such that $y=S(1,x_1)-S(1,x_2) \in Y$ by the very Definition \ref{def_bi} of $(X,Y)$-continuity.
 \end{proof}

 Applying Proposition \ref{th-bi} to the reaction-diffusion system \eqref{eq} we obtain 
\begin{theorem}
 Let conditions \eqref{f1}-\eqref{kappa} hold and $g\in L^2(D)$. Then the reaction-diffusion system \eqref{eq} in any space dimension $N\geq 1$ has a global attractor $\A$ in $L^2(D)$, and 
 \begin{itemize}
 \item[(i)] 
 the attractor $\A$ is a compact subset of $L^p(D) $ but attracts bounded subsets of $L^2(D)$ in the topology of any $L^\gamma(D) $ for $\gamma\geq 2$;
 \item[(ii)] for any $z_0\in \A$ the translation $\A-z_0$ of $\A$ is a compact subset of any $L^\gamma(D) $, $\gamma\geq 2$;
 \item[(iii)] if $g=0$, then the global attractor $\A$ is a compact subset of any $L^\gamma(D) $, $\gamma\geq 2$;
 \item[(iv)] if, moreover, condition \eqref{f_add} holds, then the attractor $\A$ as well as its translation $\A-z_0$ is a compact set in $H_0^1(D)$. 
 \end{itemize}
\end{theorem}
 \begin{proof}
 Theorem \ref{lem} shows that the semigroup generated by \eqref{eq} is $(L^2,L^\gamma)$-continuous for any $\gamma\geq 2$, and by Lemma \ref{lem8.1} the attractor is bounded in $L^p$ with $g\in L^2$ and is bounded in any $L^\gamma$ when $g=0$. In addition, with \eqref{f_add}, by Theorem \ref{lem2} the system is $(L^2,H_0^1)$-continuous and the attractor $\A$ is bounded in $H_0^1$ by Lemma \ref{lem5}. Hence, the theorem follows from 
 Proposition \ref{th-bi}.
 \end{proof}

 \subsection{Finite fractal dimensions} 
 As already noted in Remark \ref{rem},
 Theorem \ref{lem} indicates a smoothing property of the semigroup of \eqref{eq}, which is known important in estimating the upper bounds of the dimensions of a global attractor as well as in constructing an exponential attractor and further estimating its attracting rate, see, e.g., \cite{eden94,miranville08chapter,caraballo17dcds}, etc. 
 In the following we study the fractal dimension of the global attractor $\A$ and its translation $\A-z_0$ as an example to make use of the new smoothing properties. 
 
 Recall that the fractal dimension \cite{robinson11} of a compact subset $A$ of a Banach space $X$ is defined by 
 \[
 dim_F(A;X) =\liminf_{\varepsilon\to 0^+} \frac{\log N_{\varepsilon }(A;X)} {-\log \varepsilon} ,
 \]
 where $ N_\varepsilon(A;X)$ denotes the minimal number of $\varepsilon$-balls in $X$ necessary to cover $A$.
 
 For subsets of $X $ that are not included in $Y$ it mathematically makes no sense to talk about the covers by balls in $Y$, but it is possible to study the $\varepsilon$-nets under the metric of $Y$. 
 
 \begin{definition} \rm
 Let $A$ be a nonempty subset of $X$ and $\varepsilon>0$. An $\varepsilon$-net of $A$ under the metric of $Y$, called shortly an $\varepsilon|_Y$-net, is a subset $ E$ of $A$ satisfying that for any $a\in A$ there exists an $a_0\in E$ such that $\| a-a_0\|_Y<\varepsilon$. 
 \end{definition}
 Note that not all the subsets of $X$ have $\varepsilon|_Y$-nets. If $A\subset Y$, then an $\varepsilon|_Y$-net $ E $ corresponds to a cover by $\varepsilon$-balls in $Y$ centered at every element of $E$.

 \begin{lemma} \label{lem-a}
 Let $A$ be a nonempty subset of $X$ and $x\in A$. Suppose that $ \mathcal M$ is a mapping from $X$ to $X$ (not necessarily taking values in $Y$) which is $(X,Y)$-smoothing
 \be \label{cond}
 \|\mathcal M(x_1)-\mathcal M(x_2)\|_Y\leq L \|x_1-x_2\|_X^\delta,\quad \forall x_1,x_2\in X, \|x_1-x_2\|\leq 1,
 \ee 
 for some constants $L>0$ and $\delta >0$. 
 Then 
 \begin{itemize}
\item[(i)]
 $A$ has an $\varepsilon|_X$-net $ E$ iff $A-x$ has an $\varepsilon|_X$-net $ E-x$;
 \item[(ii)] if for any $x_1,x_2\in A$ we have $x_1-x_2\in Y$, then 
 $A$ has an $\varepsilon|_Y$-net $E$ iff $A-x$ has an $\varepsilon|_Y$-net $ E-x$;
 \item[(iii)] for any $\varepsilon\in (0,1]$, $A$ has an $\varepsilon|_X$-net $ E$ implies that $\mathcal M(A)$ has an $L\varepsilon^\delta |_Y$-net $\mathcal M(E) $.
\end{itemize}
 
 \end{lemma}
 \begin{proof} (i) and (ii) are straightforward by definition, and we prove 
 (iii). Since $\varepsilon\in (0,1]$ and $A$ has an $\varepsilon|_X$-net $ E$, for any $\mathcal M(a)\in \mathcal M(A)$ there exists an $a_0\in E $ such that $\|a-a_0\|_X<\varepsilon \leq 1$, so by \eqref{cond}
 \[ 
 \|\mathcal M(a)-\mathcal M(a_0)\|_Y \leq L\|a-a_0\|_X^\delta < L\varepsilon^\delta,
 \]
 i.e., $ \mathcal M(E)$ is indeed an $L\varepsilon^\delta |_Y$-net of $\mathcal M(A)$.
 \end{proof}

 For a nonempty set $E $ we denote by $^\# \! E$ the cardinality of $E$, where $^\#\! E=\infty$ is allowed. For a subset $A$ of $X$ that has finite $\varepsilon|_Y$-nets, by $\mathcal N_{\varepsilon,Y}(A)$ we denote the $\varepsilon|_Y$-net of $A$ that has minimal cardinality, i.e., if $E$ is another $\varepsilon|_Y$-net, then $^\#\! \mathcal N_{\varepsilon,Y}(A) \leq {} ^\#\! E$.
 Then applying Lemma \ref{lem-a} to the reaction-diffusion equation \eqref{eq} we obtain
 \begin{theorem} Suppose that conditions \eqref{f1}-\eqref{kappa} hold, $g\in L^2(D)$, and that $\A$ is the finite dimensional global attractor of \eqref{eq} in $L^2(D)$. Then 
 \begin{itemize}
\item[(i)]
 $\A$ is a finite dimensional compact subset of $L^p(D)$ with 
 \ben
 dim_F(\A; L^p(D)) \leq \frac p2 dim_F(\A;L^2(D)) ;
 \ee
 \item[(ii)] for any $z_0\in \A$ and $\gamma\geq 2$, the translation $\A-z_0$ of the attractor is a finite dimensional compact subset of $L^{\gamma}(D)$ with 
 \ben
 dim_F(\A-z_0; L^{\gamma}(D)) \leq \frac {\gamma}2 dim_F(\A;L^2(D)) ;
 \ee
 \item[(iii)] if, moreover, condition \eqref{f_add} holds, then the global attractor $\A$ is a finite dimensional compact subset of $H_0^1(D)$ with
 \ben
 dim_F(\A; H_0^1(D)) \leq (p-1) dim_F(\A;L^2(D)) .
 \ee
\end{itemize} 
 \end{theorem}
 \begin{proof}
 We prove (ii), and (i) and (iii) are concluded analogously by Theorem \ref{lem} and Theorem \ref{lem2}, respectively.
 By Theorem \ref{lem} and Remark \ref{rem}, for some constant $c>0$ 
 \[
 \| S(1, u_{0,1} )-S(1,u_{0,2}) \|_{\gamma} \leq c \| u_{0,1}-u_{0,2}\|^{\frac 2{\gamma}} ,\quad \forall \|u_{0,1}-u_{0,2}\|\leq 1 .
 \]
 Hence, by Lemma \ref{lem-a} with $\mathcal M=S(1,\cdot)$ and $X=L^2$, $Y=L^{\gamma}$, $\delta=\frac 2 {\gamma}$, we have 
 \ben
 dim_F(\A-z_0;Y) & =\liminf_{\varepsilon\to 0^+} \frac{\log{} ^\# \! \mathcal N_{\varepsilon,Y}(\A-z_0)} {-\log \varepsilon} \quad \text{(since $
 \A-z_0\subset Y$) } \\
 & =\liminf_{\varepsilon\to 0^+} \frac{\log {} ^\# \! \mathcal N_{\varepsilon,Y}(\A) } {-\log \varepsilon} \quad \text{ (by Lemma \ref{lem-a} (ii))} \\
 & =\liminf_{\varepsilon\to 0^+} \frac{\log {} ^\# \! \mathcal N_{L\varepsilon^\delta,Y}(S(1,\A) )} {-\log L\varepsilon^\delta} \quad \text{ (since $S(1,\A)=\A$)} \\
 & \leq \liminf_{\varepsilon\to 0^+} \frac{\log {}^\# \! \mathcal N_{\varepsilon, X }( \A) } {- \log L\varepsilon^\delta } \quad \text{ (by Lemma \ref{lem-a} (iii))} \\
 &=\frac 1\delta dim_F(\A;X) \quad \text{(since $\A\subset X$)} 
 \ee
 as desired. The proof is complete. 
 \end{proof}

 \section*{Acknowledgements} 
 Cui was partially funded by NSFC Grant 11801195 and the Fundamental Research Funds for the Central Universities 5003011026. Kloeden was partially supported by the Chinese NSF grant 11571125.
 Zhao was partially supported by CTBU Grant 1751041.

 \vspace{3mm}
 \today
 \vspace{3mm}
\end{document}